\documentclass[11pt]{article}

\usepackage{amsmath,amsthm,amssymb,graphics}
\usepackage{mathrsfs}
\usepackage{fullpage}
\usepackage{amsfonts}
\usepackage{epsfig}
\usepackage{color}
\newtheorem{thm}{Theorem}[section]
\newtheorem{cor}[thm]{Corollary}
\newtheorem{prop}[thm]{Proposition}
\newtheorem{lem}[thm]{Lemma}
\theoremstyle{definition}
\newtheorem{defn}[thm]{Definition}

\newtheorem{rem}[thm]{Remark}
\theoremstyle{remark} \numberwithin{equation}{section}

\newcommand{\R}{\mathbb{R}}

\newcommand{\mathsym}[1]{{}}

\begin{document}
%\doublespace
%\setlength{\baselineskip}{17pt} (for double line space)
%\renewcommand{\thefootnote} {\fnsymbol{footnote}}
\setcounter{page}{1}
\title{\textbf{Stochastic functional differential equations driven by G-Browniain motion with monotone nonlinearity}}
\author{\textbf{Faiz Faizullah$^{\footnote{E-mail:
faiz \b{} math@yahoo.com/faiz \b{}
math@ceme.nust.edu.pk/g.a.faizullah@swansea.ac.uk}}$ }
 \vspace{0.1cm}\\
 {Department of Mathematics, Swansea University, Singleton Park SA2
8PP UK}
 \\{Department of BS and H, College of E and ME, National
University}\\{ of Sciences and Technology (NUST) Pakistan}}
\maketitle
\begin{abstract} By using the Picard iteration scheme, this
article establishes the existence and uniqueness theory for
solutions to stochastic functional differential equations driven by
G-Browniain motion. Assuming the monotonicity conditions, the
boundedness and existence-uniqueness results of solutions have been
derived. The error estimation between Picard approximate solution
$y^k(t)$ and exact solution $y(t)$ has been determined. The $L^2_G$
and exponential estimates have been obtained. The theory has been
further generalized to weak monotonicity conditions. The existence,
uniqueness and exponential estimate under the weak monotonicity
conditions have been inaugurated.

\end{abstract}
\section{Introduction}
The existence and uniqueness theory for solutions to stochastic
dynamical systems is always a significant theme and has received a
huge attention, for instance see \cite{bl,cgr,f1,j,mc,m2,o,rf}. In
several evolution phenomena, the hereditary properties such as
time-lag, time-delay or after-effect arise in the variables
\cite{bg,lf,myz,mpbf,z}. This naturally leads us to use stochastic
functional differential equations which take into consideration the
history of the system.
 Assuming the growth and
Lipschitz conditions, Ren et al. \cite{rbs} and Faizullah \cite{f6}
gave the existence-uniqueness results for solutions to stochastic
functional differential equations in the G-framework (G-SFDEs). The
idea was generalized by Faizullah to G-SFDEs with non-Lipschitz
conditions \cite{f2}. He further extended the theory to develop the
$p$th moment estimates for the solutions to G-SFDEs \cite{f3,f4}.
The existence-uniqueness theory for neutral stochastic functional
differential equations driven by G-Brownian motion (G-NSFDEs) was
developed by Faizullah \cite{ff} and Faizullah et al. \cite{fb}. The
exponential stability, the $p$th moment exponential estimate and
stability with markovian switching for solutions to G-NSFDEs were
respectively given by Zhu et al. \cite{zlz}, Faizullah et al.
\cite{fm} and Li et al. \cite{ly}. However, to the best of our
knowledge, no text is available on the existence, uniqueness,
exponential estimate, error estimation for Picard approximate and
exact solution of stochastic functional differential equations
driven by G-Brownian under the nonlinear monotonicity conditions.
The aim of this article is to inaugurate the mentioned unavailable
literature. In addition, the existence, uniqueness, $L^2_G$ and
exponential estimates for solutions of G-SFDEs with weak nonlinear
monotonicity conditions are studied. Let $C((-\infty,0];\R^n)$ be
the collection of continuous functions from $(-\infty,0]$ to $\R^n$,
then for a given number $q>0$ we define the phase space with fading
memory $C_q((-\infty,0];\R^n)$ by
\begin{equation*} C_q((-\infty,0];\R^n)=\{\psi\in C((-\infty,0];\R^n):\lim_{\vartheta\rightarrow -\infty}e^{q\vartheta}\psi(\vartheta)\,\, \text{exists
in}\,\, \R^n \}.
\end{equation*}
 This space is complete with norm
$\|\psi\|_q=\sup_{-\infty<\vartheta\leq
0}e^{q\vartheta}|\psi(\vartheta)|<\infty$. The space
$C_q((-\infty,0];\R^n)$ is a Banach space of bounded and continuous
functions and $C_{q_1}\subseteq C_{q_2}$ for any $0< q_1\leq
q_2<\infty$ \cite{ks,wym}. Let $\mathcal{B}(C_q)$ be the
$\sigma$-algebra generated by $C_q$ and $C_q^0=\{\psi\in C_q:
\lim_{\vartheta\rightarrow-\infty}e^{q\vartheta}\psi(\vartheta)=0\}$.
Denote by $L^2(C_q)$ (resp. $L^2(C^0_q)$) the space of all
$\mathcal{F}$-measurable $C_q$-valued (resp. $C^0_q$-valued)
stochastic processes $\psi$ such that $E\|\psi\|_q^2<\infty$. Let
$(\Omega,\mathcal{F},\mathbb{P})$ be a complete probability space,
$B(t)$ be an $n$-dimensional G-Brownian motion and
$\mathcal{F}_t=\sigma\{B(v):0\leq v\leq t\}$ be the natural
filtration. Assume that the filtration $\{\mathcal{F}_t;t\geq0\}$
assures the usual conditions. Let $\mathcal{P}$ be the set of all
probability measures on $(C_q,\mathcal{B}(C_q))$ and $L_b(C_q)$ be
the collection of all  continuous bounded functionals. Let $f:
[0,T]\times C_q((-\infty,0];\R^n)\rightarrow\R^n$, $g:[0,T]\times
C_q((-\infty,0];\R^{n})\rightarrow\R^{n\times m}$ and $h:[0,T]\times
C_q((-\infty,0];\R^n)\rightarrow\R^{n\times m}$ be Borel measurable.
Consider the following stochastic functional differential equation
driven by G-Brownian motion
\begin{equation}\label{1} dy(t) = f (t, y_{t}) dt + g ( t,y_{t}) d \langle B, B
\rangle (t) + h (t, y_{t}) dB (t),
\end{equation}
on $t\in[0,T]$ with given initial condition  $\zeta(0)\in \R^n$ and
$ y_{t} = \{ y(t + \vartheta) : -\infty < \vartheta \leq 0\}$. The
coefficients $f$, $g$ and $h$ are given functions such that for all
$y\in\R^n$,  $f(.,y),g(.,y),h(.,y)\in M^2_G((-\infty,T];\R^n)$. For
problem \eqref{1}, the initial data is given as follows.
\begin{equation}\label{i} y_0=\zeta=\Big\{\zeta(\vartheta):-\infty<\vartheta\leq0\Big\},
\end{equation}
is $\mathcal{F}_0$-measurable, $C_q((-\infty,0];\R^n)$-value random
variable such that $\zeta\in M^2_G((-\infty,0];\R^n)$.
\begin{defn} A stochastic process $y(t)\in\R^n$, $t\in(-\infty,T]$, is said to be a solution of the above equation \ref{1} with the given initial data
\eqref{i}, if
\begin{itemize}
\item[${\bf(1)}$] For all $t\in[0,T]$, $y(t)$ is $\mathcal {F}_t$-adapted and
continuous.
\item[${\bf(2)}$] The coefficients $f(t,y_t)\in\mathcal {L}^1([0,T];\R^n)$, $g(t,y_t)\in \mathcal
{L}^1([0,T];\R^{n\times m})$ and $h(t,y_t)\in \mathcal
{L}^2([0,T];\R^{n\times m})$
\item[${\bf(3)}$] For each $t\in[0,T]$, $y(t) = \zeta+\int_{0}^tf (v, y_{v}) dv +\int_{0}^t g (v, y_{v}) d \langle B, B
\rangle (v) + \int_{0}^th (v, y_{v}) dB (v)$ q.s.
\end{itemize}
\end{defn}
The rest of the paper is organized as follows. Section 2 is devoted
to the basic notions and results required for the subsequent
sections of this article. Section 3 presents the boundedness of
solutions and contains the existence-uniqueness results with
monotone nonlinearity conditions for G-SFDEs. The error estimation
for Picard approximate solution $y^k(t)$ and exact solution $y(t)$
is determined in section 4. Section 5 gives the $L^2_G$ and
exponential estimates for the unique solution of G-SFDEs. With weak
monotonicity conditions, section 6 studies the existence and
uniqueness while section 7 the $L^2_G$ and exponential estimates for
G-SFDEs.

\section{Preliminaries}
Assume that $\mathcal {H}$ be a space of real valued functions
defined on a given non-empty set $\Omega$ and let $(\Omega,\mathcal
{H},\hat{\mathbb{E}})$ be a sublinear expectation space. Let
$\Omega$ denotes the space of all $\R^n$-valued continuous paths
$(w(t))_{t\geq 0}$ with $w(0)=0$ equipped with the distance
\begin{equation*}\rho(w^1,w^2)=\sum_{i=1}^{\infty}\frac{1}{2^i}\Big(\max_{t\in[0,i]}|w^1(t)-w^2(t)|\wedge1\Big).\end{equation*}
Let for any $w\in\Omega$ and $t\geq 0$, $B(t)=B(t,w)=w(t)$ be the
canonical process.  For any fixed $T\in[0,\infty)$, set
\begin{equation*}L_{ip}(\Omega_T)=\Big\{\phi(B(t_1),B(t_2),...,B(t_n)):n\geq1,t_{1},t_{2},...,t_{n}\in[0,T],\phi\in
C_{b.Lip}(\R^{n\times m}))\Big\},\end{equation*}
 where $C_{b.Lip}(\R^{n\times m})$ is a space of bounded Lipschitz functions and  $L_{ip}(\Omega_t)\subseteq L_{ip}(\Omega_T)$ for $t\leq
 T$,
$L_{ip}(\Omega)=\cup_{n=1}^{\infty}L_{ip}(\Omega_n)$. The completion
of $L_{ip}(\Omega)$ under the
 Banach norm $\hat{\mathbb{E}}[|.|^p]^{\frac{1}{p}}$, $p\geq 1$ is denoted by
 $L^p_{G}(\Omega)$, where
$L_{G}^p(\Omega_t)\subseteq L_{G}^p(\Omega_T)\subseteq
L_{G}^p(\Omega)$ for $0\leq t\leq T <\infty.$ Generated by the
canonical process $\{B(t)\}_{t\geq 0}$, the filtration is given by
$\mathcal {F}_t=\sigma\{B(s), 0\leq s \leq t\}$, $\mathcal
{F}=\{\mathcal {F}_t\}_{t\geq 0}$. Let $\pi_T=\{t_0,t_1,...,t_N\}$,
$0\leq t_0\leq t_1\leq...\leq t_N\leq\infty$ be a partition of
$[0,T].$ For all $N\geq 1$, $0=t_0<t_1<...<t_{N}=T$ and
$i=0,1,...,N-1$, define the space $M^{p,0}_G([0,T])$, $p\geq 1$ of
simple processes as
\begin{equation}\label{p1}M^{p,0}_G([0,T])=\Big\{\eta_t(w)=\sum_{i=0}^{N-1}\xi_{t_i}(w)I_{[t_i,t_{i+1}]}(t);\,\,\xi_{t_i}(w)\in L_G^p(\Omega_{t_{i}})\Big\}.\end{equation}
Let $M_{G}^{p}(0,T),$ $p\geq 1$ denotes the completion of
$M_{G}^{p,0}(0,T)$ with the norm given below
\begin{equation*}\|\eta\|=\Big\{\int_0^T\hat{\mathbb{E}}[|\eta(v)|^p]dv\Big\}^{1/p}.\end{equation*}
\begin{defn} For $\eta_t\in
M_G^{2,0}(0,T)$, the G-It\^{o}'s integral $I(\eta)$ is defined by
\begin{align*}
I(\eta)=\int_0^T\eta(v)dB^a(v)=\sum_{i=0}^{N-1}\xi_i\Big(B^a({t_{i+1}})-B^a({t_i})\Big).\end{align*}
A mapping $I:M^{2,0}_{G}(0,T)\mapsto L^2_G(\mathcal{F}_T)$  can be
continuously extended to $I:M^2_G(0,T)\mapsto L^2_G(\mathcal{F}_T)$
and for $\eta\in M^2_G(0,T)$ the G-It\^o integral is still defined
by
\begin{align*}
\int_0^T\eta(v)dB^a(v)=I(\eta).\end{align*}
\end{defn}
\begin{defn} The G-quadratic variation process $\{\langle
B^a\rangle(t)\}_{t\geq0}$ of G-Brownian motion is defined by
\begin{align*}\begin{split}&
\langle
B^a\rangle(t)=\lim_{N\rightarrow\infty}\sum_{i=0}^{N-1}\Big(B^a({t_{i+1}^N})-B^a({t_{i}^N})\Big)^2={B^a(t)}^2-2\int_0^tB^a(v)dB^a(v),\\&
\end{split}\end{align*}
which is an increasing process with $\langle B^a\rangle(0)=0$ and
for any $0\leq s\leq t$,
\begin{align*} \langle B^a\rangle(t)-\langle B^a\rangle(s)\leq
\sigma_{aa^\tau} (t-s).
\end{align*}
\end{defn}
Assume that $a, \hat{a} \in\R^n$ be two given vectors. Then the
mutual variation process of $B^a$ and $B^{\hat{a}}$ is defined by
$\langle B^a,B^{\hat{a}}\rangle=\frac{1}{4}[\langle
B^a+B^{\hat{a}}\rangle(t)-\langle B^a-B^{\hat{a}}\rangle(t)]$. A
mapping $H_{0,T}:M^{0,1}_{G}(0,T)\mapsto L^2_G(\mathcal{F}_T)$  is
defined by
\begin{align*} H_{0,T}(\eta)=\int_0^T\eta(v)d\langle B^a\rangle(v)= \sum_{i=0}^{N-1}\xi_i\Big(\langle
B^a\rangle_({t_{i+1}})-\langle B^a\rangle({t_{i}})\Big),
\end{align*}
which can be continuously extended to $M^1_G(0,T)$ and for $\eta\in
M^1_G(0,T)$ this is still defined by
\begin{align*}
\int_0^T\eta(v)d\langle B^a\rangle(v)=H_{0,T}(\eta).
\end{align*}
The G-It\^{o} integral and its quadratic variation process satisfy
the following properties \cite{p3,wzl}.
\begin{prop}
\begin{itemize}
\item[${\bf(1)}$] $\hat{\mathbb{E}}[\int_0^T\eta(v)dB(v)]=0,\,\,\, \textit{for all}\,\,\, \eta\in M^p_G(0,T)$.
\item[${\bf(2)}$]
$\hat{\mathbb{E}}[(\int_0^T\eta(v)dB(v))^2]=\hat{\mathbb{E}}[\int_0^T\eta^2(v)\langle
B, B\rangle(v)]\leq \bar{\sigma}^2
\mathbb{E}[\int_0^T\eta^2(v)dv],\,\,\, \textit{for all}\,\,\,
\eta\in M^2_G(0,T)$.
\item[${\bf(3)}$] $\hat{\mathbb{E}}[\int_0^T|\eta(v)|^pdv]\leq \int_0^T\hat{\mathbb{E}}|\eta(v)|^pdv,\,\,\,
\textit{for all}\,\,\, \eta\in M^p_G(0,T)$.
\end{itemize}
\end{prop}
The concept of G-capacity and lemma \ref{l2} can be found in
\cite{dhp}.
\begin{defn}
Let $\mathcal {B}(\Omega)$ be a Borel $\sigma$-algebra of $\Omega$
and $\mathcal{P}$ be a collection of all probability measures on
$(\Omega, \mathcal {B}(\Omega)$. Then the G-capacity denoted by
$\hat{C}$ is defined as the following
\begin{equation*}\hat{C}(A)=\sup_{\mathbb{P}\in\mathcal{P}}\mathbb{P}(A),\end{equation*}
where set $A\in\mathcal {B}(\Omega)$.
\end{defn}
\begin{defn}
A set $A\in\mathcal {B}(\Omega)$ is said to be polar if its capacity
is zero i.e. $\hat{C}(A)=0$ and a property holds quasi-surely (q.s)
if it holds outside a polar set.
\end{defn}
\begin{lem}\label{l2} Let $y\in L^p$ and $\hat{\mathbb{E}}|y|^p<\infty$. Then for each
$\alpha>0,$ the G-Markov inequality is defined by
\begin{equation*}\hat{C}(|y|>\alpha)\leq \frac{\hat{\mathbb{E}}[|y|^p]}{\alpha}.\end{equation*}
\end{lem}
For the proof of the following lemmas \ref{l3} and \ref{l4} see
\cite{g}.
\begin{lem}\label{l3} Let $p\geq 2$, $\eta\in M_G^2(0,T)$, $a\in\R^n$ and $y(t)=\int_0^t
\eta(v)dB^a(v)$. Then there exists a continuous modification
$\bar{y}(t)$ of $y(t)$, that is, on some $\bar{\Omega}\subset\Omega$
with $\hat{C}(\bar{\Omega}^c)=0$ and for all $t\in[0,T]$,
$\hat{C}(|y(t)-\bar{y}|\neq 0)=0$ such that
\begin{equation*}
\hat{\mathbb{E}}\Big[\sup_{s\leq v\leq
t}|\bar{y}(v)-\bar{y}(s)|^p\Big]\leq
\hat{K}\sigma_{aa^{\tau}}^{\frac{p}{2}}\hat{\mathbb{E}}\Big(\int_s^t|\eta(v)|^2dv\Big)^{\frac{p}{2}},
\end{equation*}
where $0<\hat{K}<\infty$ is a positive constant.
\end{lem}
\begin{lem}\label{l4} Let $p\geq 1$, $\eta\in M_G^p(0,T)$ and $a,\hat{a}\in\R^n$, then there exists a continuous modification
$\bar{y}^{a,\hat{a}}(t)$ of  $y^{a,\hat{a}}(t)=\int_0^t \eta(v)d
\langle B^{a}, B^{\hat{a}} \rangle (v)$ such that for $0\leq s\leq
v\leq t\leq T$,
\begin{equation*}
\hat{\mathbb{E}}\Big[\sup_{0\leq s\leq v\leq
t}|\bar{y}^{a,\hat{a}}(v)-\bar{y}^{a,\hat{a}}(s)|^p\Big]\leq
\Big(\frac{1}{4}\sigma_{(a+\hat{a})(a-\hat{a})^\tau}\Big)^{p}(t-s)^{p-1}\hat{\mathbb{E}}\int_s^t|\eta(v)|^pdv,
\end{equation*}
\end{lem}
The following lemma can be found in \cite{m}.
\begin{lem}\label{l7} Let $a,b\in\R^n$ and
$\hat{c}>0.$ Then
\begin{align*}
|a+b|^2\leq (1+\hat{c})|a|^2+(1+\hat{c}^{-1})|b|^2.
\end{align*}
\end{lem}
The following lemma is borrowed from \cite{f7}.
\begin{lem}\label{Lf} Let  $p\geq 2$ and $\lambda< pq$. Then for any $\zeta\in C_q((-\infty,0];\R^n)$,
\begin{equation*}\label{}
\hat{\mathbb{E}}\|y_t\|^p_q\leq e^{-\lambda t}
\hat{\mathbb{E}}\|\zeta\|_q^p+\hat{\mathbb{E}}\Big[\sup_{0<v\leq
t}|y(v)|^p\Big]. \end{equation*}
\end{lem}
For the following lemma see \cite{m,wc}.
\begin{lem}\label{fl}  Let
$\kappa(.):\R^+\rightarrow\R^+$ be a concave non-decreasing
continuous function satisfying $\kappa(0)=0$ and $\kappa(y)>0$ for
$y>0$. Assume that $\mu(t)\geq0$ for all $0\leq t\leq T <\infty,$
satisfies
\begin{equation*} \mu(t)\leq c+\int_{0}^t \varphi(s)\kappa(\mu(s))ds,\end{equation*} where $c$ is a positive real number and
$\varphi:[0,T]\rightarrow \R^{+}$. Then the following properties
hold.
\begin{itemize}
\item[${\bf(i)}$] If $c=0,$ then $\mu(t)=0,$ $t\in[0,T].$
\item[${\bf(ii)}$] If  $c>0,$ we define  $\omega(t)=\int_{0}^t
\frac{1}{\kappa(v)}dv,$ for $t\in[0,T]$, then
\begin{equation*}\mu(t)\leq \omega^{-1}(\omega(c)+\int_{0}^t \varphi(s)dv),\end{equation*}
where $\omega^{-1}$ is the inverse function of $\omega.$
\end{itemize}
\end{lem}

\section{The G-SFDEs with monotone nonlinearity}
Consider equation \eqref{1} with the corresponding initial data
\eqref{i}. Let the coefficients $f$, $g$ and $h$ of \eqref{1}
satisfy the following conditions.
\begin{itemize}
\item[${\bf(H_1)}$] For any $y\in C_q((-\infty,0];\R^n)$,
 there exists a constant $K$ such that,
\begin{equation}\label{c11}
2\langle y(0),f(t,y)\rangle \vee 2\langle y(0),g(t,y)\rangle \vee
|h(t,y)|^2\leq K(1+\|y\|_q^2),\,\,\,\,\,\,t\in[0,T].
\end{equation}
\item[${\bf(H_2)}$] For any $y,z \in C_q((-\infty,0];\R^n)$, there exists a constant  $\hat{K}$ such
that,
\begin{equation}\label{A1}
\begin{split}
& 2 \langle z (0) - y (0), f ( t,z) - f ( t,y ) \rangle\vee 2
\langle z (0) - y (0), g (t, z) - g ( t,y ) \rangle \\&\vee | h(t,z)
- h(t,y)|^2 \leq \hat{K}\| z  - y\|_{q}^{2},\,\,\,\,\,\,t\in[0,T].
\end{split}
\end{equation}
\end{itemize}
First of all, let us  see the following useful lemma.
\begin{lem}\label{Lf1} Let $y(t)$ be any solution of problem \eqref{1} with
initial data \eqref{i} and $\hat{\mathbb{E}}\|\zeta\|_q^2<\infty$.
 Assume that assumption $H_1$ holds, then
\begin{equation*}
\hat{\mathbb{E}}\Big[\sup_{-\infty< v\leq T}|y(v)|^2\Big]\leq
\hat{\mathbb{E}}\|\zeta\|_q^2+C_1e^{c_3T},\end{equation*} where
$C_1=c_3T+ (2+c_3\lambda^{-1})\hat{\mathbb{E}}\|\zeta\|_q^2$,
$c_3=2K(1+2c_1+2c^2_2)$, $c_1$ and $c_2$ are positive constants.
\end{lem}
\begin{proof} Applying
the G-It$\hat{o}$ formula to $|y(t)|^2$, taking the G-expectation on
both sides, using properties of G-It\^o integral, lemma \ref{l3} and
lemma \ref{l4}, there exist $c_1>0$ and $c_2>0$ such that for any
$t\in[0, T]$,
\begin{equation*}\begin{split}
\hat{\mathbb{E}}\Big[\sup_{0\leq v\leq t}|y(v)|^2\Big]&\leq
\hat{\mathbb{E}}\|\zeta\|_q^2+ \hat{\mathbb{E}}\Big[\sup_{0\leq
v\leq t}\int_0^t2\langle y(v),f(v,y_v)\rangle
dv\Big]+\hat{\mathbb{E}}\Big[\sup_{0\leq v\leq t}\int_0^t2\langle
y(v), h(v,y_v)\rangle dB(v)\Big]\\&+\hat{\mathbb{E}}\Big[\sup_{0\leq
v\leq t}\int_0^t[2\langle y(v), g(v,y_v)\rangle+|h(v,y_v)|^2]d
\langle B, B \rangle (v)\Big]\\& \leq \hat{\mathbb{E}}\|\zeta\|_q^2+
\hat{\mathbb{E}}\int_0^t2\langle y(v), f(v,y_v)\rangle
dv+2c_2\hat{\mathbb{E}}\Big[\int_0^t|\langle y(v),
h(v,y_v)\rangle|^2dv\Big]^{\frac{1}{2}}\\&+
c_1\hat{\mathbb{E}}\int_0^t[2\langle y(v),
g(v,y_v)\rangle+|h(v,y_v)|^2]dv\\& \leq
\hat{\mathbb{E}}\|\zeta\|_q^2+ \hat{\mathbb{E}}\int_0^t2\langle
y(v),f(v,y_v)\rangle dv+
2c^2_2\hat{\mathbb{E}}\int_0^t|h(v,y_v)|^2dv
+\frac{1}{2}\hat{\mathbb{E}}\Big[\sup_{0\leq v\leq t}
|y(v)|^2\Big]\\&+ c_1\hat{\mathbb{E}}\int_0^t[2\langle
y(v),g(v,y_v)\rangle+|h(v,y_v)|^2]dv.
\end{split}\end{equation*}
By using assumption $H_1$ and lemma \ref{Lf}, it follows
\begin{equation*}\begin{split}
\hat{\mathbb{E}}\Big[\sup_{0\leq v\leq t}|y(v)|^2\Big]&
 \leq 2\hat{\mathbb{E}}\|\zeta\|_q^2+
2(1+2c_1+2c^2_2)K\hat{\mathbb{E}}\int_0^t(1+\|y\|_q^2)dv\\& \leq
2\hat{\mathbb{E}}\|\zeta\|_q^2+2(1+2c_1+2c^2_2)K(T+\lambda^{-1}\hat{\mathbb{E}}\|\zeta\|_q^2)\\&
+2(1+2c_1+2c^2_2)K\hat{\mathbb{E}}\int_0^t\Big[\sup_{0\leq v\leq
t}|y(v)|^2\Big]dv.
\end{split}\end{equation*}
In virtue of the Grownwall inequality, we have
\begin{equation}\label{3}
\hat{\mathbb{E}}\Big[\sup_{0\leq v\leq t}|y(v)|^2\Big]\leq
C_1e^{c_3t},\end{equation} where $C_1=c_3T+
(2+c_3\lambda^{-1})\hat{\mathbb{E}}\|\zeta\|_q^2$ and
$c_3=2K(1+2c_1+2c^2_2)$. Noticing that
\begin{equation*}
\hat{\mathbb{E}}\Big[\sup_{-\infty< v\leq t}|y(v)|^2\Big]\leq
\hat{\mathbb{E}}\|\zeta\|_q^2+\hat{\mathbb{E}}\Big[\sup_{0\leq v\leq
t}|y(v)|^2\Big],
\end{equation*}
we get
\begin{equation*}
\hat{\mathbb{E}}\Big[\sup_{-\infty< v\leq t}|y(v)|^2\Big]\leq
\hat{\mathbb{E}}\|\zeta\|_q^2+C_1e^{c_3t}.
\end{equation*}
By letting $t=T$, the proof of the required assertion completes.
\end{proof}
\begin{rem} Lemma \ref{Lf1} states that the solution $y(t)$ is
bounded, in particular, $y(t)\in M^2_G((-\infty,T];\R^n)$.
\end{rem}
Next under the assumptions $H_1$ and $H_2$, we prove the
existence-uniqueness results for the G-SFDE \eqref{1} with the given
initial data \eqref{i} in the phase space with fading memory
$C_q((-\infty,T];\R^n)$. First, we derive the uniqueness of
solutions.
\begin{defn} A solution $y(t)$ of problem \eqref{1} with the initial data \eqref{i} is said to be unique if it is
indistinguishable from any other solution $z(t)$, that is,
\begin{equation*}\hat{\mathbb{E}}\Big[\sup_{-\infty< v\leq t}|z(v)-y(v)|^2\Big]=0,\end{equation*}
quasi-surely.
\end{defn}
\begin{thm} Let assumption $H_2$ holds. Then \eqref{1} has a unique
solution, if exists.
\end{thm}
\begin{proof} Let \eqref{1} has two solutions say $y(t)$ and $z(t)$ with the same initial
data. By virtue of lemma \ref{Lf1}, $y(t),z(t)\in
M^2_G((-\infty,T];\R^n)$. Define $\Lambda(t)=z(t)-y(t)$,
$\hat{f}(t)=f(t,z_t)-f(t,y_t)$, $\hat{g}(t)=g(t,z_t)-g(t,y_t)$ and
$\hat{h}(t)=h(t,z_t)-h(t,y_t)$. Applying the G-It$\hat{o}$ formula
to $|\Lambda(t)|^2$, taking the G-expectation on both sides, using
properties of G-It\^o integral, lemma \ref{l3} and lemma \ref{l4},
there exist $c_1>0$ and $c_2>0$ such that for any $t\in[0, T]$,
\begin{equation*}\begin{split}
\hat{\mathbb{E}}\Big[\sup_{0\leq v\leq t}|\Lambda(v)|^2\Big]&\leq
\hat{\mathbb{E}}\Big[\sup_{0\leq v\leq
t}\int_0^t2\langle\Lambda(v),\hat{f}(v)\rangle dv\Big]+
\hat{\mathbb{E}}\Big[\sup_{0\leq v\leq t}\int_0^t2\langle\Lambda(v),
\hat{h}(v)\rangle dB(v)\Big]\\& +\hat{\mathbb{E}}\Big[\sup_{0\leq
v\leq
t}\int_0^t[2\langle\Lambda(v),\hat{g}(v)\rangle+|\hat{h}(v)|^2d
\langle B, B \rangle (v)\Big]\\& \leq
\hat{\mathbb{E}}\int_0^t2\langle\Lambda(v),\hat{f}(v)\rangle
dv+2c_2\hat{\mathbb{E}}\Big[\int_0^t|\langle\Lambda(v)
,\hat{h}(v)\rangle|^2dv\Big]^{\frac{1}{2}}\\&+
c_1\hat{\mathbb{E}}\int_0^t[2\langle\Lambda(v),\hat{g}(v)\rangle+|\hat{h}(v)|^2]dv\\&
\leq \hat{\mathbb{E}}\int_0^t2\langle\Lambda(v),\hat{f}(v)\rangle dv
+ 2c^2_2\hat{\mathbb{E}}\int_0^t|\hat{h}(v)|^2dv
+\frac{1}{2}\hat{\mathbb{E}}\Big[\sup_{0\leq v\leq t}
|\Lambda(v)|^2\Big]\\& +
c_1\hat{\mathbb{E}}\int_0^t[2\langle\Lambda(v),\hat{g}(v)\rangle+|\hat{h}(v)|^2]dv.
\end{split}\end{equation*}
In view of assumption $H_2$, it follows
\begin{equation*}\begin{split}
\hat{\mathbb{E}}\Big[\sup_{0\leq v\leq t}|z(v)-y(v)|^2\Big]& \leq
2(1+2c_1+2c^2_2)\hat{K}\hat{\mathbb{E}}\int_0^t\| z  - y\|_{q}^{2}dv
\end{split}\end{equation*}
Noticing that initial data of $z(t)$ and $y(t)$ is same, lemma
\ref{Lf} yields
\begin{equation*}\hat{\mathbb{E}}\| z - y\|_{q}^{2}\leq
\hat{\mathbb{E}}\Big[\sup_{0\leq v\leq t}|z(v)-y(v)|^2\Big],
\end{equation*}
which on substituting in the above last inequality gives
\begin{equation*}\begin{split}
\hat{\mathbb{E}}\Big[\sup_{0\leq v\leq t}|z(v)-y(v)|^2\Big]& \leq
2(1+2c_1+2c^2_2)\hat{K}\int_0^t\hat{\mathbb{E}}\Big[\sup_{0\leq
v\leq t}|z(v)-y(v)|^2\Big]dv.
\end{split}\end{equation*}
By using the Grownwall inequality, we derive
\begin{equation*}
\hat{\mathbb{E}}\Big[\sup_{0\leq v\leq t}|z(v)-y(v)|^2\Big]=0,
\end{equation*}
because the initial data of $y(t)$ and $z(t)$ is same, it follows
\begin{equation*}
\hat{\mathbb{E}}\Big[\sup_{-\infty <v\leq t}|z(v)-y(v)|^2\Big]=0.
\end{equation*}
This shows that for $t\in(-\infty,T]$, $y(t)=z(t)$ quasi-surely. The
proof of uniqueness is complete.
\end{proof}
To prove the existence of solutions we use the Picard iteration
scheme. For $t\in[0,T]$, define $y^0(t)=\zeta(0)$ and $y^0_0=\zeta$.
For each $k=1,2,...$, set $y^k_0=\zeta$ and for $t\in[0,T]$, define
the Picard iterations,
\begin{equation}\label{f1} y^k(t) =\zeta(0)+\int_0^t f ( v,y^{k-1}_{v}) dv + \int_0^tg ( v,y^{k-1}_{v}) d \langle B, B
\rangle (v) +\int_0^t h (v, y^{k-1}_{v}) dB (v).
\end{equation}
It is obvious that $y^0(t)\in M^2_G((-\infty,T];\R^n)$ and by
induction for each $k=1,2,...$, $y^k(t)\in M^2_G((-\infty,T];\R^n)$,
which is derived in the following lemma \ref{Lf2}.
\begin{lem}\label{Lf2} Let assumption $H_1$ holds and
$\hat{\mathbb{E}}\|\zeta\|_q^2<\infty$. Then
\begin{equation}\label{f4}
\hat{\mathbb{E}}\Big[\sup_{0\leq v\leq t}|y^k(v)|^2\Big]\leq
C_2e^{c_3T},\end{equation} where
$C_2=c_3T+(2+c_3(T+\lambda^{-1}))\hat{\mathbb{E}}\|\zeta\|_q^2$,
$c_3=2K(1+2c_1+2c^2_2)$ and $c_1$, $c_2$ are positive constants.
\end{lem}
\begin{proof}Applying
the G-It$\hat{o}$ formula to $|y^k(t)|^2$, taking the G-expectation
on both sides, using properties of G-It\^o integral, lemma \ref{l3}
and lemma \ref{l4}, there exist $c_1>0$ and $c_2>0$ such that for
any $t\in[0, T]$
\begin{equation*}\begin{split}
\hat{\mathbb{E}}\Big[\sup_{0\leq v\leq t}|y^k(v)|^2\Big]&\leq
\hat{\mathbb{E}}\Big[\sup_{0\leq v\leq t}\int_0^t2\langle
y^k(v),f(v,y^{k-1}_v)\rangle
dv\Big]+\hat{\mathbb{E}}\Big[\sup_{0\leq v\leq t}\int_0^t2\langle
y^k(v), h(v,y^{k-1}_v)\rangle dB(v)\Big]\\&+
\hat{\mathbb{E}}\Big[\sup_{0\leq v\leq t}\int_0^t[2\langle
y^k(v),g(v,y^{k-1}_v)\rangle+|h(v,y^{k-1}_v)|^2d \langle B, B
\rangle (v)\Big]\\& \leq \hat{\mathbb{E}}\int_0^t2\langle
y^k(v),f(v,y^{k-1}_v)\rangle dv +
2c^2_2\hat{\mathbb{E}}\int_0^t|h(v,y^{k-1}_v)|^2dv
+\frac{1}{2}\hat{\mathbb{E}}\Big[\sup_{0\leq v\leq t}
|y^k(v)|^2\Big]\\&+ c_1\hat{\mathbb{E}}\int_0^t[2\langle
y^k(v),g(v,y^{k-1}_v)\rangle+|h(v,y^{k-1}_v)|^2]dv.
\end{split}\end{equation*}
By using assumption $H_1$ and lemma \ref{Lf}, it follows
\begin{equation*}\begin{split}
\hat{\mathbb{E}}\Big[\sup_{0\leq v\leq t}|y^k(v)|^2\Big]&
 \leq 2\hat{\mathbb{E}}\|\zeta\|_q^2+
2(1+2c_1+4c^2_2)K\hat{\mathbb{E}}\int_0^t(1+\|y^{k-1}\|_q^2)dv\\&
\leq
2\hat{\mathbb{E}}\|\zeta\|_q^2+2(1+2c_1+2c^2_2)K(T+\lambda^{-1}\hat{\mathbb{E}}\|\zeta\|_q^2)\\&
+2(1+2c_1+2c^2_2)K\hat{\mathbb{E}}\int_0^t\Big[\sup_{0\leq v\leq
t}|y^{k-1}(v)|^2\Big]dv.
\end{split}\end{equation*}
Noticing that
\begin{equation*}\begin{split}\max_{1\leq k\leq n}\hat{\mathbb{E}}\Big[\sup_{0\leq v\leq
t}|y^{k-1}(v)|^2\Big]&\leq
\max\Big\{\hat{\mathbb{E}}\|\zeta\|_q^2,\max_{1\leq k\leq
n}\hat{\mathbb{E}}\Big[\sup_{0\leq v\leq
t}|y^{k}(v)|^2\Big]\Big\}\\&\leq
\hat{\mathbb{E}}\|\zeta\|_q^2+\max_{1\leq k\leq
n}\hat{\mathbb{E}}\Big[\sup_{0\leq v\leq t}|y^{k}(v)|^2\Big],
\end{split}\end{equation*}
we derive
\begin{equation*}\begin{split}
\max_{1\leq k\leq n}\hat{\mathbb{E}}\Big[\sup_{0\leq v\leq
t}|y^k(v)|^2\Big]& \leq
2\hat{\mathbb{E}}\|\zeta\|_q^2+2K(1+2c_1+2c^2_2)(T+(T+\lambda^{-1})\hat{\mathbb{E}}\|\zeta\|_q^2)\\&
+2K(1+2c_1+2c^2_2)\int_0^t\max_{1\leq k\leq
n}\hat{\mathbb{E}}\Big[\sup_{0\leq v\leq t}|y^{k}(v)|^2\Big]dv,
\end{split}\end{equation*}
In virtue of the Grownwall inequality and noting that $n$ is
arbitrary, we have
\begin{equation*}
\hat{\mathbb{E}}\Big[\sup_{0\leq v\leq t}|y^k(v)|^2\Big]\leq
2\Big[\hat{\mathbb{E}}\|\zeta\|_q^2+K(1+2c_1+2c^2_2)(T+(T+\lambda^{-1})\hat{\mathbb{E}}\|\zeta\|_q^2)\Big]e^{2K(1+2c_1+2c^2_2)t}.\end{equation*}
Finally, by using the fact
\begin{equation*}
\hat{\mathbb{E}}\Big[\sup_{-\infty< v\leq t}|y^k(v)|^2\Big]\leq
\hat{\mathbb{E}}\|\zeta\|_q^2+\hat{\mathbb{E}}\Big[\sup_{0\leq v\leq
t}|y^k(v)|^2\Big],
\end{equation*}
 and letting $t=T$, the proof of assertion \eqref{f4} completes.
\end{proof}
\begin{thm}\label{thm1} Let assumptions $H_1$ and $H_2$ hold. Then problem
\eqref{1} with the given initial data \eqref{i} has a solution
$y(t)\in M^2_G((-\infty,T];\R^n)$.
\end{thm}
\begin{proof} Consider the Picard iteration sequence $\{y^k(t);t\geq0\}$
defined by \eqref{f1}. Then from \eqref{f1}, we have
\begin{equation*} y^1(t) =\zeta(0)+\int_0^t f (v, y^{0}_{v}) dv + \int_0^tg (v, y^{0}_{v}) d \langle B, B
\rangle (v) +\int_0^t h (v, y^{0}_{v}) dB (v).
\end{equation*}
Applying the G-It\^o formula to $|y^1(t)|^2$ and using similar
arguments as earlier, it follows
\begin{equation*}\begin{split}
\hat{\mathbb{E}}\Big[\sup_{0\leq v\leq t}|y^1(v) |^2\Big]&\leq
\hat{\mathbb{E}}\|\zeta\|^2+ \hat{\mathbb{E}}\Big[\sup_{0\leq v\leq
t}\int_0^t2\langle y^1(v), f(v,y^{0}_{v})\rangle dv\Big]
+\hat{\mathbb{E}}\Big[\sup_{0\leq v\leq t}\int_0^t2\langle y^1(v) ,
h(v,y^{0}_{v})\rangle B(v)\Big]\\&+\hat{\mathbb{E}}\Big[\sup_{0\leq
v\leq t}\int_0^t[2\langle y^1(v)
,g(v,y^{0}_{v})\rangle+|h(v,y^{0}_{v})|^2d \langle B, B \rangle
(v)\Big]\\& \leq \hat{\mathbb{E}}\|\zeta\|^2+
\hat{\mathbb{E}}\int_0^t2\langle y^1(v), f(v,y^{0}_{v})\rangle dv+
2c^2_2\hat{\mathbb{E}}\int_0^t|h(v,y^{0}_{v})|^2dv
+\frac{1}{2}\hat{\mathbb{E}}\Big[\sup_{0\leq v\leq t} |y^1(v)
|^2\Big]\\&+c_1\hat{\mathbb{E}}\int_0^t[2\langle y^1(v),
g(v,y^{0}_{v})\rangle+|h(v,y^{0}_{v})|^2]dv.
\end{split}\end{equation*}
By assumption $H_2$, we have
\begin{equation*}\begin{split}
\hat{\mathbb{E}}\Big[\sup_{0\leq v\leq t}|y^1(v)|^2\Big]&
\leq2\hat{\mathbb{E}}\|\zeta\|^2_q+
2K(1+2c_1+2c_2^2)\hat{\mathbb{E}}\int_0^t(1+\|y^0_v\|^2_q) dv\\&
\leq 2\hat{\mathbb{E}}\|\zeta\|^2_q+
2K(1+2c_1+2c_2^2)T+2K(1+2c_1+2c_2^2)\hat{\mathbb{E}}\int_0^t[e^{-\lambda
t}\|\zeta\|^2_q+\|\zeta\|^2_q] dv\\& \leq 2K(1+2c_1+2c_2^2)T +2[1+
K(1+2c_1+2c_2^2)(T+\lambda^{-1})]\hat{\mathbb{E}}\|\zeta\|^2_q.
\end{split}\end{equation*}
By using lemma \ref{l7}, it follows
\begin{equation*}\begin{split}
\hat{\mathbb{E}}\Big[\sup_{0\leq v\leq t}|y^1(v) -y^0(v)|^2\Big]\leq
(1+\hat{c})\hat{\mathbb{E}}\Big[\sup_{0\leq v\leq
t}|y^1(v)|^2\Big]+(1+\hat{c}^{-1})\hat{\mathbb{E}}\|\zeta\|^2_q \leq
L,
\end{split}\end{equation*}
where $L=(1+\hat{c})[c_3T +[2+
c_3(T+\lambda^{-1})]\hat{\mathbb{E}}\|\zeta\|^2_q]+(1+\hat{c}^{-1})\hat{\mathbb{E}}\|\zeta\|^2_q$.
In a similar fashion as above, from \eqref{f1} we obtain
\begin{equation*}\begin{split}
\hat{\mathbb{E}}\Big[\sup_{0\leq v\leq t}|y^2(v) -y^1(v)|^2\Big]&
\leq \hat{\mathbb{E}}\int_0^t2\langle y^2(v)
-y^1(v),f(v,y^{1}_{v})-f(v,y^{0}_{v})\rangle dv\\&+
2c^2_2\hat{\mathbb{E}}\int_0^t|h(v,y^{1}_{v})-h(v,y^{0}_{v})|^2dv
+\frac{1}{2}\hat{\mathbb{E}}\Big[\sup_{0\leq v\leq t} |y^2(v)
-y^1(v)|^2\Big]\\&+ c_1\hat{\mathbb{E}}\int_0^t[2\langle y^2(v)
-y^1(v),g(v,y^{1}_{v})-g(v,y^{0}_{v})\rangle+|h(v,y^{1}_{v})-h(v,y^{0}_{v})|^2]dv.
\end{split}\end{equation*}
By using assumption $H_2$ and lemma \ref{Lf}, it follows
\begin{equation*}\begin{split}
\hat{\mathbb{E}}\Big[\sup_{0\leq v\leq t}|y^2(v) -y^1(v)|^2\Big]&
\leq
2(1+2c_1+2c^2_2)\hat{K}\hat{\mathbb{E}}\int_0^t\|y^{1}_{v}-y^{0}_{v}\|dv\\&
\leq 2(1+2c_1+2c^2_2)\hat{K}\int_0^t\hat{\mathbb{E}}\Big[\sup_{0\leq
v\leq t}|y^{1}(v)-y^{0}(v)|^2\Big]dv\\&\leq 2(1+2c_1+2c^2_2)\hat{K}
Lt.
\end{split}\end{equation*}
Similarly, we derive
\begin{equation*}\begin{split}
\hat{\mathbb{E}}\Big[\sup_{0\leq v\leq t}|y^3(v) -y^1(v)|^2\Big]&
\leq
2(1+2c_1+2c^2_2)\hat{K}\hat{\mathbb{E}}\int_0^t\|y^{2}_{v}-y^{1}_{v}\|dv\\&
\leq 2(1+2c_1+2c^2_2)\hat{K}\int_0^t\hat{\mathbb{E}}\Big[\sup_{0\leq
v\leq t}|y^{2}(v)-y^{1}(v)|^2\Big]dv\\&\leq
L\frac{1}{2!}[2(1+2c_1+2c^2_2)\hat{K}]^2 t^2.
\end{split}\end{equation*}
Continuing this procedure, we get
\begin{equation*}\begin{split}
\hat{\mathbb{E}}\Big[\sup_{0\leq v\leq t}|y^4(v)
-y^3(v)|^2\Big]&2(1+2c_1+2c^2_2)\hat{K}\hat{\mathbb{E}}\int_0^t\|y^{2}_{v}-y^{1}_{v}\|dv\\&
\leq 2(1+2c_1+2c^2_2)\hat{K}\int_0^t\hat{\mathbb{E}}\Big[\sup_{0\leq
v\leq t}|y^{3}(v)-y^{2}(v)|^2\Big]dv\\& \leq
L\frac{1}{3!}[2(1+2c_1+2c^2_2)\hat{K}]^3 t^3.
\end{split}\end{equation*}
Now for all $k\geq 0$, we have to verify that
\begin{equation}\label{2} \hat{\mathbb{E}}\Big[\sup_{0\leq v\leq
t}|y^{k+1}(t)-y^{k}(t)|^2\Big] \leq \frac{L[Mt]^k}{k!},
\end{equation}
where $L=(1+\hat{c})[c_3T +[2+
c_3(T+\lambda^{-1})]\hat{\mathbb{E}}\|\zeta\|^2_q]+(1+\hat{c}^{-1})\hat{\mathbb{E}}\|\zeta\|^2_q$
and $M=2(1+2c_1+2c^2_2)\hat{K}$. To prove that \eqref{2} is true for
all $k\geq 0$, we use the procedure of mathematical induction as
follows.  For $k=0$, it has been proved above. Next suppose that
\eqref{2} holds for some $k\geq 0$, we have to show that it holds
for $k+1$. Define $\Lambda^{k+2,k+1}(t)=y^{k+2}(t)-y^{k+1}(t)$,
$\hat{f}^{k+1,k}(t)=f(y^{k+1}_t)-f(y^{k}_t)$,
$\hat{g}^{k+1,k}(t)=g(y^{k+1}_t)-g(y^{k}_t)$ and
$\hat{h}^{k+1,k}(t)=h(y^{k+1}_t)-h(y^{k}_t)$. By using the
G-It$\hat{o}$ formula, lemma \ref{l3} and lemma \ref{l4} for any
$t\in[0, T]$, we obtain
\begin{equation*}\begin{split}
&\hat{\mathbb{E}}\Big[\sup_{0\leq v\leq
t}|\Lambda^{k+2,k+1}(v)|^2\Big]\\&\leq
\hat{\mathbb{E}}\Big[\sup_{0\leq v\leq
t}\int_0^t2\langle\Lambda^{k+2,k+1}(v),\hat{f}^{k+1,k}(v)\rangle
dv\Big]+\hat{\mathbb{E}}\Big[\sup_{0\leq v\leq
t}\int_0^t2\langle\Lambda^{k+2,k+1}(v), \hat{h}^{k+1,k}(v)\rangle
dB(v)\Big]\\&+ \hat{\mathbb{E}}\Big[\sup_{0\leq v\leq
t}\int_0^t[2\langle\Lambda^{k+2,k+1}(v),\hat{g}^{k+1,k}(v)\rangle+|\hat{h}^{k+1,k}(v)|^2d
\langle B, B \rangle (v)\Big]\\& \leq
\hat{\mathbb{E}}\int_0^t2\langle\Lambda^{k+2,k+1}(v),\hat{f}^{k+1,k}(v)\rangle
dv+ 2c^2_2\hat{\mathbb{E}}\int_0^t|\hat{h}^{k+1,k}(v)|^2dv
+\frac{1}{2}\hat{\mathbb{E}}\Big[\sup_{0\leq v\leq t}
|\Lambda^{k+2,k+1}(v)|^2\Big]\\&+
c_1\hat{\mathbb{E}}\int_0^t[2\langle\Lambda^{k+2,k+1}(v),\hat{g}^{k+1,k}(v)\rangle+|\hat{h}^{k+1,k}(v)|^2]dv.
\end{split}\end{equation*}
In view of assumption $H_2$, it follows
\begin{equation*}\begin{split}
\hat{\mathbb{E}}\Big[\sup_{0\leq v\leq
t}|y^{k+2}(v)-y^{k+1}(v)|^2\Big]& \leq
2(1+2c_1+2c^2_2)\hat{K}\hat{\mathbb{E}}\int_0^t\|
y^{k+1}-y^{k}\|_{q}^{2}dv.
\end{split}\end{equation*}
By using lemma \ref{Lf}, we obtain
\begin{equation*}\hat{\mathbb{E}}\| y^{k+1}-y^{k}\|_{q}^{2}\leq
\hat{\mathbb{E}}\Big[\sup_{0\leq v\leq
t}|y^{k+1}(v)-y^{k}(v)|^2\Big],
\end{equation*}
which on substituting in the above last inequality gives
\begin{equation*}\begin{split}
\hat{\mathbb{E}}\Big[\sup_{0\leq v\leq
t}|y^{k+2}(v)-y^{k+1}(v)|^2\Big]& \leq
2(1+2c_1+2c^2_2)\hat{K}\int_0^t\hat{\mathbb{E}}\Big[\sup_{0\leq
v\leq t}|y^{k+1}(v)-y^{k}(v)|^2\Big]dv\\& \leq
2(1+2c_1+2c^2_2)\hat{K}\int_0^t\frac{L[Mt]^k}{k!}dv
=\frac{L[Mt]^{k+1}}{(k+1)!}.
\end{split}\end{equation*}
This implies that \eqref{2} holds for $k+1$. Thus by induction
\eqref{2} holds for all $k\geq0$. Next by using lemma \ref{l2} we
derive
\begin{equation*}\begin{split}
\hat{C}\Big\{\sup_{0\leq s\leq
T}|y^{k+1}(t)-y^{k}(t)|^2>\frac{1}{2^k}\Big\}&\leq
2^k\hat{\mathbb{E}}\Big[\sup_{0\leq t\leq
T}|y^{k+1}(t)-y^{k}(t)|^2\Big] \leq\frac{L[2Mt]^k}{k!}
\end{split}\end{equation*}
Since $\sum_{k=0}^{\infty}\frac{L[2Mt]^k}{k!}<\infty$, the
Borel-Cantelli lemma gives that for almost all $w$ there exists a
positive integer $k_0=k_0(w)$ such that
\begin{equation*}
\sup_{0\leq s\leq T}|y^{k+1}(t)-y^{k}(t)|^2\leq\frac{1}{2^k},\,\,as
\,\,k\geq k_0.
\end{equation*}
It implies that q.s. the partial sums
\begin{equation*}
y^0(t)+\sum_{i=0}^{k-1}[y^{k+1}(t)-y^{k}(t)]=y^k(t),
\end{equation*}
are convergent uniformly on $t\in(-\infty,T]$. Denote the limit by
$y(t)$. Then the sequence $\{y^{k}(t)\}_{t\geq0}$ converges
uniformly to $y(t)$ on $t\in(-\infty,T]$. Clearly, $y(t)$ is
continuous and $\mathcal{F}_t$-adapted because
$\{y^{k}(t)\}_{t\geq0}$ is continuous and $\mathcal{F}_t$-adapted.
Also, from \eqref{2}, we can see that $\{y^k(t):k\geq1\}$ is a
cauchy sequence in $L^2_G$. Hence $y^k(t)$ converges to $y(t)$ in
$L^2_G$, that is,
\begin{equation*}
\hat{\mathbb{E}}|y^{k}(t)-y(t)|^2\rightarrow 0,
\,\,as\,\,k\rightarrow\infty.
\end{equation*}
Taking limits $k\rightarrow\infty$, from \eqref{f4} in lemma
\ref{Lf2} we obtain
\begin{equation*}
\hat{\mathbb{E}}\Big[\sup_{0\leq s\leq t}|y(s)|^2\Big]\leq
C_2e^{c_3T}.
\end{equation*}
To show that the sequence of solution maps $\{y^k_t:n\geq1\}$ is
convergent in $L^2_G$, by using lemma \ref{Lf} we get
\begin{equation*}\begin{split}
\hat{\mathbb{E}}\|y^k_t(\zeta)-y_t(\xi)\|^2_q&\leq e^{-\lambda
t}\hat{\mathbb{E}}\|\zeta-\xi\|_q^2
+\hat{\mathbb{E}}\Big[\sup_{0<s\leq t}|y^k(s)-y(s)|^2\Big].
\end{split}\end{equation*}
Since $y^k(t)$ and $y(t)$ have the same initial data and $y^n(t)$ is
convergent to $y(t)$, we therefore have
\begin{equation*}\begin{split}
\hat{\mathbb{E}}\|y^n_t(\zeta)-y_t(\zeta)\|^2_q\rightarrow
0\,\,as\,\,n\rightarrow\infty.
\end{split}\end{equation*}
This implies that the sequence $\{y^n_t\}_{t\geq0}$  converges to
$y_t$ in $L^2_G$ and we have
\begin{equation*}\begin{split}  &\int_0^tg (v, y^{n}_{v}) dv\rightarrow\int_0^tg ( v,y_{v}) dv,\,\,in\,\,L^2_G,\\&
\int_0^th ( v,y^{n}_{v}) d \langle B, B \rangle (v)\rightarrow
\int_0^th ( v,y_{v}) d \langle B, B \rangle (v),\,\,in\,\,L^2_G,\\&
\int_0^th (v, y^{n}_{v}) dB (v)\rightarrow\int_0^th ( v,y_{v}) dB
(v), \,\,in\,\,L^2_G.
\end{split}\end{equation*}
For $t\in[0,T]$, by taking limits $n\rightarrow \infty$ in
\eqref{f1} we derive
\begin{equation*}\lim_{k\rightarrow\infty} y^k(t) =\zeta(0)+\int_0^t\lim_{k\rightarrow\infty} f (v, y^{k-1}_{v}) dv
+ \int_0^t\lim_{k\rightarrow\infty}g (v, y^{k-1}_{v}) d \langle B, B
\rangle (v) +\int_0^t\lim_{k\rightarrow\infty} h (v, y^{k-1}_{v}) dB
(v),
\end{equation*}
which yields that
\begin{equation*} y(t) =\zeta(0)+\int_0^t f (v, y_{v}) dv
+ \int_0^tg ( v,y_{v}) d \langle B, B \rangle (v) +\int_0^t h (v,
y_{v}) dB (v),
\end{equation*}
$t\in[0,T]$. This shows that $y(t)$ is the solution of \eqref{1}.
The proof of existence is complete.
\end{proof}

\section{Error Estimation}
The procedure of the proof of above existence results demonstrates
how to construct the Picard sequence $\{y^k(t);t\geq 0\}$ and gain
the accurate solution $y(t)$. We now show the estimate of error for
Picard approximate solution $y^k(t)$ and exact solution $y(t)$.

\begin{thm} Let $y(t)$ be the unique solution of problem \eqref{1} with initial data \eqref{i} and $y^k(t)$ be defined by \eqref{f1}.
Assume that the assumptions $H_1$ and $H_2$ hold. Then for all
$k\geq 1$,
\begin{equation*}
\hat{\mathbb{E}}\Big[\sup_{0\leq v\leq T}|y^k(v)-y(v)|^2\Big] \leq
\frac{L[Mt]^k}{k!} e^{MT}, \end{equation*} where $L=(1+\hat{c})[c_3T
+[2+
c_3(T+\lambda^{-1})]\hat{\mathbb{E}}\|\zeta\|^2_q]+(1+\hat{c}^{-1})\hat{\mathbb{E}}\|\zeta\|^2_q$
and $M=2(1+2c_1+2c^2_2)\hat{K}$.
\end{thm}
\begin{proof}
We define $\Lambda(t)=y^k(t)-y(t)$,
$\hat{f}(t)=f(t,y^k_t)-f(t,y_t)$, $\hat{g}(t)=g(t,y^k_t)-g(t,y_t)$
and $\hat{h}(t)=h(t,y^k_t)-h(t,y_t)$. Then in a similar fashion as
earlier we obtain
\begin{equation*}\begin{split}
\hat{\mathbb{E}}\Big[\sup_{0\leq v\leq t}|\Lambda(v)|^2\Big]& \leq
\hat{\mathbb{E}}\int_0^t2\langle\Lambda(v),\hat{f}(v)\rangle dv +
2c^2_2\hat{\mathbb{E}}\int_0^t|\hat{h}(v)|^2dv
+\frac{1}{2}\hat{\mathbb{E}}\Big[\sup_{0\leq v\leq t}
|\Lambda(v)|^2\Big]\\& +
c_1\hat{\mathbb{E}}\int_0^t[2\langle\Lambda(v),\hat{g}(v)\rangle+|\hat{h}(v)|^2]dv.
\end{split}\end{equation*}
In view of assumption $H_2$, it follows
\begin{equation*}\begin{split}
\hat{\mathbb{E}}\Big[\sup_{0\leq v\leq t}|y^k(v)-y(v)|^2\Big]& \leq
2(1+2c_1+2c^2_2)\hat{K}\hat{\mathbb{E}}\int_0^t\| y^k  -
y\|_{q}^{2}dv\\&
\leq2(1+2c_1+2c^2_2)\hat{K}\int_0^t\hat{\mathbb{E}}\Big[\sup_{0\leq
v\leq t}|y^k(v)-y(v)|^2\Big]dv\\& \leq
M\int_0^t\hat{\mathbb{E}}\Big[\sup_{0\leq v\leq
t}|y^k(v)-y^{k-1}(v)|^2\Big]dv+M\int_0^t\hat{\mathbb{E}}\Big[\sup_{0\leq
v\leq t}|y^{k-1}(v)-y(v)|^2\Big]dv
\end{split}\end{equation*}
By using \eqref{2}, we derive
\begin{equation*}\begin{split}
\hat{\mathbb{E}}\Big[\sup_{0\leq v\leq t}|y^k(v)-y(v)|^2\Big]& \leq
\frac{L[Mt]^k}{k!} +M\int_0^t\hat{\mathbb{E}}\Big[\sup_{0\leq v\leq
t}|y^{k-1}(v)-y(v)|^2\Big]dv
\end{split}\end{equation*}
Finally, the grownwall inequality yields,
\begin{equation*}\begin{split}
\hat{\mathbb{E}}\Big[\sup_{0\leq v\leq t}|y^k(v)-y(v)|^2\Big]& \leq
\frac{L[Mt]^k}{k!} e^{Mt},
\end{split}\end{equation*}
for $t\in[0,T]$. By letting $t=T$, we get the desired expression.
The proof stands completed.
\end{proof}
Consider, the following stochastic differential equation driven by
G-Brownian motion $\{B(t);t\geq 0\}$
\begin{equation}\label{n1} dy(t) = f (t, y(t)) dt +g ( t,y(t)) d \langle B, B
\rangle (t)+ h ( t,y(t)) dB (t),
\end{equation}
with initial data $X_0\in\R^n$ and $f, g, h:\Omega\times[0,T]\times
\R^n\rightarrow \R^n$. In \cite{wzl}, Wei et al. have derived the
exponential estimate for the unique solution of problem \eqref{n1}
under the monotone type conditions. However, till now it is not
known that problem \eqref{n1} under the monotone type conditions has
a unique solution. By using the procedure of this article, one can
prove the following existence-uniqueness result.
\begin{cor}  Under the hypothesis $H_1$ and $H_2$, problem
\eqref{n1} admits a unique solution $y(t)\in M^2_G([0,,T];\R^n)$.
\end{cor}

\section{The exponential estimate}
First, we assume that under hypothesis $H_1$ and $H_2$ problem
\eqref{1} with the given initial data \eqref{i} admits a unique
solution on $[0,\infty)$. Then we derive the $L^2_G$ and exponential
estimates as follows.
\begin{lem} Let $y(t)$ be a unique solution of \eqref{1} on
$t\in[0,\infty)$ and $\hat{\mathbb{E}}\|\zeta\|^2_q<\infty$. Then
under the hypothesis $H_1$ for all $t\geq0,$
\begin{equation*}
\hat{\mathbb{E}}\Big[\sup_{-\infty< v\leq t}|y(v)|^2\Big]\leq C_3
e^{c_3t},\end{equation*} where $C_3= c_3T+
(3+c_3\lambda^{-1})\hat{\mathbb{E}}\|\zeta\|_q^2$,
$c_3=2(1+2c_1+2c^2_2)K$ and $c_1$, $c_2$ are positive constants.
\end{lem}
\begin{proof} By straightforward calculations in a similar procedure
used in lemma \ref{Lf1}, we obtain
\begin{equation}\label{n2}
\hat{\mathbb{E}}\Big[\sup_{0\leq v\leq t}|y(v)|^2\Big] \leq c_3T+
(2+c_3\lambda^{-1})\hat{\mathbb{E}}\|\zeta\|_q^2
+c_3\hat{\mathbb{E}}\int_0^t\Big[\sup_{0\leq v\leq
t}|y(v)|^2\Big]dv.
\end{equation}
Noticing that
\begin{equation*}
\hat{\mathbb{E}}\Big[\sup_{-\infty< v\leq t}|y(v)|^2\Big]\leq
\hat{\mathbb{E}}\|\zeta\|_q^2+ \hat{\mathbb{E}}\Big[\sup_{0\leq
v\leq t}|y(v)|^2\Big],
\end{equation*}
it follows
\begin{equation*}\begin{split}
\hat{\mathbb{E}}\Big[\sup_{-\infty< v\leq t}|y(v)|^2\Big]&\leq
\hat{\mathbb{E}}\|\zeta\|_q^2+ c_3T+
(2+c_3\lambda^{-1})\hat{\mathbb{E}}\|\zeta\|_q^2
+c_3\hat{\mathbb{E}}\int_0^t\Big[\sup_{0\leq v\leq
t}|y(v)|^2\Big]dv\\& \leq  c_3T+
(3+c_3\lambda^{-1})\hat{\mathbb{E}}\|\zeta\|_q^2
+c_3\hat{\mathbb{E}}\int_0^t\Big[\sup_{-\infty\leq v\leq
t}|y(v)|^2\Big]dv\\&
\end{split}\end{equation*}
Finally, by using the Grownwall inequality, it follows
\begin{equation*}
\hat{\mathbb{E}}\Big[\sup_{-\infty< v\leq t}|y(v)|^2\Big]\leq[c_3T+
(3+c_3\lambda^{-1})\hat{\mathbb{E}}\|\zeta\|_q^2]e^{c_3t}.\end{equation*}
The proof is complete.
\end{proof}

\begin{thm}\label{thm2} Let $y(t)$ be a unique solution of \eqref{1} on
$t\in[0,\infty)$ and $\hat{\mathbb{E}}\|\zeta\|^2_q<\infty$. Then
for all $t\geq 0$,
\begin{equation*}
\lim_{t\rightarrow\infty}\sup \frac{1}{t}log|y(t)|\leq \alpha,
\end{equation*}
where $\alpha=K(1+2c_1+4c^2_2)$.
\end{thm}

\begin{proof}
From \eqref{n2}, we derive
\begin{equation}\label{n}
\hat{\mathbb{E}}\Big[\sup_{0\leq v\leq t}|y(v)|^2\Big]\leq
C_1e^{c_3t},\end{equation} where $C_1=c_3T+
(2+c_3\lambda^{-1})\hat{\mathbb{E}}\|\zeta\|_q^2$. By virtue of the
above result \eqref{n}, for each $m=1,2,3,...,$ we have
\begin{equation*}
\hat{\mathbb{E}}\Big[\sup_{m-1\leq t\leq m}|y(t)|^2\Big]\leq
C_1e^{c_3m}.\end{equation*} For any $\epsilon>0$, by using lemma
\ref{l2} we get
\begin{equation*}\begin{split}
\hat{C}\Big\{w:\sup_{m-1\leq t\leq
m}|y(t)|^2>e^{(c_3+\epsilon)m}\Big\}&\leq
\frac{\hat{\mathbb{E}}\Big[\sup_{m-1\leq t\leq
m}|y(t)|^2\Big]}{e^{(c_3+\epsilon)m}}\\&
\leq\frac{C_1e^{c_3m}}{e^{(c_3+\epsilon)m}} =C_1e^{-\epsilon m}.
\end{split}\end{equation*}
But for almost all $w\in\Omega$, the Borel-Cantelli lemma yields
that there exists a random integer $m_0=m_0(w)$ so that
\begin{equation*}
\sup_{m-1\leq t\leq m}|y(t)|^2\leq e^{(c_3+\epsilon)m},\,\,
\textit{as}\,\,m\geq m_0,\end{equation*} which implies
\begin{equation*}\begin{split}
\lim_{t\rightarrow\infty}\sup \frac{1}{t}log|y(t)|&\leq
\frac{2K(1+2c_1+2c^2_2)+\epsilon}{2}=
K(1+2c_1+2c^2_2)+\frac{\epsilon}{2},
\end{split}\end{equation*}
but $\epsilon$ is arbitrary and the above result reduces to
\begin{equation*}
\lim_{t\rightarrow\infty}\sup \frac{1}{t}log|y(t)|\leq \alpha,
\end{equation*}
where $\alpha=K(1+2c_1+2c^2_2)$. The proof is complete.
\end{proof}

\section{Existence and uniqueness with weak monotonicity}
In this section, we assume that the coefficients of problem
\eqref{1} with the initial data \eqref{i} satisfy the following weak
nonlinear monotonicity conditions.
\begin{itemize}
\item[${\bf(A_1)}$] There exists a non-decreasing and concave function $\kappa(.):\R^+\rightarrow \R^+$
 with $\kappa(0)=0,$ $\kappa(z)>0$ for $z>0$ and
$\int_{0+}\frac{dz}{\kappa(z)}=\infty$ such that for any $y,z\in
C_q((-\infty,0];\R^n)$,
\begin{equation*}\begin{split}
&2\Big\langle z(0)-y(0),f(z,t)-f(y,t)\Big\rangle\vee 2\Big\langle
z(0)-y(0),g(z,t)-g(y,t)\Big\rangle\\&\vee |h(z,t)-h(y,t)|^2\leq
\kappa(\|z-y\|_q^2),\,\,\,t\in[0,T],
\end{split}\end{equation*}
where for all $z\geq0$ and $a,b\in\R^+$, $\kappa(z)\leq a+b z$.
 \item[${\bf(A_{2})}$] For any
 $t\in[0,T]$ and $f(t,0), g(t,0), h(t,0)\in L^2,$
 there exists a positive constant $\tilde{K}$ such that
\begin{equation*}|f(t,0)|^2\vee|g(t,0)|^2\vee|h(t,0)|^2\leq
\tilde{K}.\end{equation*} \end{itemize} We now prove two useful
lemmas. They will be used in the upcoming existence-uniqueness
results.
\begin{lem}\label{Lf3} Let assumptions $A_1$ and $A_2$ hold and
$\hat{\mathbb{E}}\|\zeta\|_q^2<\infty$. Then
\begin{equation*}
\hat{\mathbb{E}}\Big[\sup_{0\leq v\leq t}|y^k(v)|^2\Big]\leq C,
\end{equation*}
where $C=C_4e^{\hat{C_4}T}$, $C_4=(1+\hat{c}b\lambda^{-1}+\hat{c}bT)
\hat{\mathbb{E}}\|\zeta\|_q^2 +\hat{c}(\tilde{K}+a)T$,
$\hat{C_4}=(2+2c_1+\hat{c}b)$, $\hat{c}=2(1+2c_1+2c^2_2)$ and
$c_1,c_2$ are positive constants.
\end{lem}
\begin{proof} Consider the Picard iteration sequence $\{y^k(t);t\geq0\}$
defined by \eqref{f1}. By using the G-It$\hat{o}$ formula, lemma
\ref{l3} and lemma \ref{l4} for any $t\in[0, T]$, we derive
\begin{equation*}\begin{split}
&\hat{\mathbb{E}}\Big[\sup_{0\leq v\leq t}|y^k(v)|^2\Big]\\&\leq
\hat{\mathbb{E}}\|\zeta\|_q^2+ \hat{\mathbb{E}}\Big[\sup_{0\leq
v\leq t}\int_0^t2\langle y^k(v),f(v,y^{k-1}_v)\rangle
dv\Big]+\hat{\mathbb{E}}\Big[\sup_{0\leq v\leq t}\int_0^t2\langle
y^k(v), h(v,y^{k-1}_v)\rangle dB(v)\Big]\\&+
\hat{\mathbb{E}}\Big[\sup_{0\leq v\leq t}\int_0^t[2\langle
y^k(v),g(v,y^{k-1}_v)\rangle+|h(v,y^{k-1}_v)|^2d \langle B, B
\rangle (v)\Big]\\& \leq \hat{\mathbb{E}}\|\zeta\|_q^2+
\hat{\mathbb{E}}\int_0^t2\langle y^k(v),f(v,y^{k-1}_v)\rangle dv +
2c^2_2\hat{\mathbb{E}}\int_0^t|h(v,y^{k-1}_v)|^2dv
+\frac{1}{2}\hat{\mathbb{E}}\Big[\sup_{0\leq v\leq t}
|y^k(v)|^2\Big]\\&+ c_1\hat{\mathbb{E}}\int_0^t[2\langle
y^k(v),g(v,y^{k-1}_v)\rangle+|h(v,y^{k-1}_v)|^2]dv.
\end{split}\end{equation*}
By using assumptions $A_1$ and $A_2$, it follows
\begin{equation*}\begin{split}
&\hat{\mathbb{E}}\Big[\sup_{0\leq v\leq t}|y^k(v)|^2\Big]\\&\leq
\hat{\mathbb{E}}\|\zeta\|_q^2+
2\hat{\mathbb{E}}\int_0^t[\kappa(\|y^{k-1}\|_q^2)+2|y^{k}(v)||f(v,0)|]
dv +
4c^2_2\hat{\mathbb{E}}\int_0^t[\kappa(\|y^{k-1}\|_q^2)+|h(v,0)|^2]dv
\\&+
2c_1\hat{\mathbb{E}}\int_0^t[\kappa(\|y^{k-1}\|_q^2)+2|y^{k}(v)||g(v,0)|+
\kappa(\|y^{k-1}\|_q^2)+|h(v,0)|^2]dv
\\&
\leq
\hat{\mathbb{E}}\|\zeta\|_q^2+2\hat{\mathbb{E}}\int_0^t[\kappa(\|y^{k-1}\|_q^2)+|y^{k}(v)|^2+|f(v,0)|^2]
dv +
4c^2_2\hat{\mathbb{E}}\int_0^t[\kappa(\|y^{k-1}\|_q^2)+|h(v,0)|^2]dv
\\&+
2c_1\hat{\mathbb{E}}\int_0^t[\kappa(\|y^{k-1}\|_q^2)+|y^{k}(v)|^2+|g(v,0)|^2+
\kappa(\|y^{k-1}\|_q^2)+|h(v,0)|^2]dv
\\&
\leq
\hat{\mathbb{E}}\|\zeta\|_q^2+\hat{c}\tilde{K}T+\hat{c}aT+2(1+c_1)\hat{\mathbb{E}}\int_0^t|y^{k}(v)|^2dv\\&+
\hat{c}b\hat{\mathbb{E}}\int_0^t\|y^{k-1}\|_q^2dv,
\end{split}\end{equation*}
where $\hat{c}=2(1+2c^2_2+2c_1)$. By virtue of lemma \ref{Lf}, we
have
\begin{equation*}\begin{split} \hat{\mathbb{E}}\Big[\sup_{0\leq
v\leq t}|y^k(v)|^2\Big]&\leq
\hat{\mathbb{E}}\|\zeta\|_q^2+\hat{c}(\tilde{K}+a)T+2(1+c_1)\hat{\mathbb{E}}\int_0^t|y^{k}(v)|^2dv\\&+
\hat{c}b\hat{\mathbb{E}}\int_0^t[\|\zeta\|_q^2e^{-\lambda
v}+\sup_{0\leq v\leq t}|y^{k-1}(v)|^2]dv\\& \leq
\hat{\mathbb{E}}\|\zeta\|_q^2+\hat{c}(\tilde{K}+a)T+2(1+c_1)\int_0^t\hat{\mathbb{E}}\Big[\sup_{0\leq
v\leq t}|y^{k}(v)|^2\Big]dv\\&+
\hat{c}b\lambda^{-1}\hat{\mathbb{E}}\|\zeta\|_q^2+
\hat{c}b\int_0^t\hat{\mathbb{E}}\Big[\sup_{0\leq v\leq
t}|y^{k-1}(v)|^2\Big]dv.
\end{split}\end{equation*}
Noticing that
\begin{equation*}\begin{split}\max_{1\leq k\leq n}\hat{\mathbb{E}}\Big[\sup_{0\leq v\leq
t}|y^{k-1}(v)|^2\Big]&\leq
\max\Big\{\hat{\mathbb{E}}\|\zeta\|_q^2,\max_{1\leq k\leq
n}\hat{\mathbb{E}}\Big[\sup_{0\leq v\leq
t}|y^{k}(v)|^2\Big]\Big\}\\&\leq
\hat{\mathbb{E}}\|\zeta\|_q^2+\max_{1\leq k\leq
n}\hat{\mathbb{E}}\Big[\sup_{0\leq v\leq t}|y^{k}(v)|^2\Big],
\end{split}\end{equation*}
we derive
\begin{equation*}\begin{split}
\max_{1\leq k\leq n}\hat{\mathbb{E}}\Big[\sup_{0\leq v\leq
t}|y^k(v)|^2\Big]&
  \leq (1+\hat{c}b\lambda^{-1}+\hat{c}bT)
\hat{\mathbb{E}}\|\zeta\|_q^2 +\hat{c}(\tilde{K}+a)T
\\&+(2+2c_1+\hat{c}b)\int_0^t\max_{1\leq k\leq
n}\hat{\mathbb{E}}\Big[\sup_{0\leq v\leq t}|y^{k}(v)|^2\Big]dv\\&
\end{split}\end{equation*}
By the Grownwall inequality, it follows
\begin{equation*}
\max_{1\leq k\leq n}\hat{\mathbb{E}}\Big[\sup_{0\leq v\leq
t}|y^k(v)|^2\Big]\leq C_4e^{\hat{C_4}t},
\end{equation*}
where $C_4=(1+\hat{c}b\lambda^{-1}+\hat{c}bT)
\hat{\mathbb{E}}\|\zeta\|_q^2 +\hat{c}(\tilde{K}+a)T$ and
$\hat{C_4}=(2+2c_1+\hat{c}b)$. But $n$ is arbitrary and letting
$t=T$, we get the desired expression. The proof is complete.
\end{proof}

\begin{lem}\label{Lf4} Let assumption $A_1$ holds. Then for
$t\in[0,T]$,
\begin{equation*}\begin{split}\hat{\mathbb{E}}\Big[\sup_{0<s\leq t}|y^{k+m}(s)-y^{k}(s)|^2\Big] &\leq
\hat{c}\int_0^t\kappa\Big(\hat{\mathbb{E}}\Big[\sup_{0<s\leq
t}|y^{k+m-1}(s)-y^{k-1}(s)|^2\Big]\Big)ds \leq \tilde{C} t,
\end{split}\end{equation*}
where $\tilde{C}=\hat{c}\kappa(4C)$, $\hat{c}=2(1+2c_1+2c^2_2)$ and
$c_1,c_2$ are positive constants.
\end{lem}
\begin{proof} Define $\Lambda^{k+m,k}(t)=y^{k+m}(t)-y^{k}(t)$,
$\hat{f}^{k+m-1,k-1}(t)=f(t,y^{k+m-1}_t)-f(t,y^{k-1}_t)$,
$\hat{g}^{k+m-1,k-1}(t)=g(t,y^{k+m-1}_t)-g(t,y^{k-1}_t)$ and
$\hat{h}^{k+m-1,k-1}(t)=h(t,y^{k+m-1}_t)-h(t,y^{k-1}_t)$. By using
the G-It$\hat{o}$ formula, lemma \ref{l3} and lemma \ref{l4} for any
$t\in[0, T]$, it follows
\begin{equation*}\begin{split}
&\hat{\mathbb{E}}\Big[\sup_{0\leq v\leq
t}|\Lambda^{k+m,k}(v)|^2\Big]\\&\leq
\hat{\mathbb{E}}\Big[\sup_{0\leq v\leq
t}\int_0^t2\langle\Lambda^{k+m,k}(v),\hat{f}^{k+m-1,k-1}(v)\rangle
dv\Big]+ \hat{\mathbb{E}}\Big[\sup_{0\leq v\leq
t}\int_0^t2\langle\Lambda^{k+m,k}(v), \hat{h}^{k+m-1,k-1}(v)\rangle
dB(v)\Big]\\& +\hat{\mathbb{E}}\Big[\sup_{0\leq v\leq
t}\int_0^t[2\langle\Lambda^{k+m,k}(v),\hat{g}^{k+m-1,k-1}(v)\rangle+|\hat{h}^{k+m-1,k-1}(v)|^2d
\langle B, B \rangle (v)\Big]\\& \leq
\hat{\mathbb{E}}\int_0^t2\langle\Lambda^{k+m,k}(v),\hat{f}^{k+m-1,k-1}(v)\rangle
dv + 2c^2_2\hat{\mathbb{E}}\int_0^t|\hat{h}^{k+m-1,k-1}(v)|^2dv
+\frac{1}{2}\hat{\mathbb{E}}\Big[\sup_{0\leq v\leq t}
|\Lambda^{k+m,k}(v)|^2\Big]\\& +
c_1\hat{\mathbb{E}}\int_0^t[2\langle\Lambda^{k+m,k}(v),\hat{g}^{k+m-1,k-1}(v)\rangle+|\hat{h}^{k+m-1,k-1}(v)|^2]dv.
\end{split}\end{equation*}
In view of assumption $A_1$ and lemma \ref{Lf3}, it follows
\begin{equation*}\begin{split}
\hat{\mathbb{E}}\Big[\sup_{0\leq v\leq
t}|y^{k+m}(v)-y^{k}(v)|^2\Big]& \leq
2(1+2c_1+2c^2_2)\hat{\mathbb{E}}\int_0^t\kappa(\| y^{k+m-1} -
y^{k-1}\|_{q}^{2})dv\\& \leq
2(1+2c_1+2c^2_2)\int_0^t\kappa\Big(\hat{\mathbb{E}}\Big[\sup_{0\leq
v\leq t}| y^{k+m-1}(v) - y^{k-1}(v)|_{q}^{2}\Big]\Big)dv\\&
\leq\hat{k}\kappa(4C)t= \tilde{C}t
\end{split}\end{equation*}
where $\tilde{C}=\hat{c}\kappa(4C)$ and
$\hat{c}=2c_2(1+2c_1+2c^2_2)$. The proof is complete.
\end{proof}
\begin{lem} Let assumptions $A_1$ and $A_2$ hold. Let $y(t)$ be a unique solution of problem \eqref{1} with initial data \eqref{i}.
 Then $y(t)$ is bounded, in particular, $y(t)\in
M^2_G((-\infty,T];\R^n)$.
\end{lem}
We omit the proof of the above lemma. It can be proved in a similar
way like lemma \ref{Lf3}. To show the existence of solution we set
that for $t\in[0,T]$,
\begin{equation}\label{e1} \mu_1(t)=\tilde{C}t,
\end{equation}
and define a recursive function as follows. For every $k,m\geq 1$,
\begin{equation}\begin{split}\label{e2}
&\mu_{k+1}(t)=\hat{k}\int_{t_0}^t\kappa(\mu_{k}(s))ds,\\&
\mu_{k,m}(t)=E[\sup_{-\tau\leq q\leq s} |y^{k+m}(q)-y^{k}(q)|^2].
\end{split}\end{equation}
Choose $T_1\in[0,T]$ such that for all $t\in[0,T_1],$
\begin{equation}\label{e3}
\hat{k}\kappa(\tilde{C}t)\leq \tilde{C}.\end{equation}
\begin{thm} Let assumptions $A_1$ and $A_2$ hold. Then equation
\ref{1} with the corresponding initial data \eqref{i} admits a
unique solution $y(t)\in M^2_G((-\infty,T];\R^n)$.
\end{thm}
\begin{proof}
We claim that for all $k\geq 1$ and any $m\geq 1$ there exists a
positive $T_1\in[0,T]$ such that
\begin{equation}\begin{split}\label{e4}
0\leq\mu_{k,m}(t)\leq \mu_{k}(t)\leq \mu_{k-1}(t)\leq...\leq
\mu_{1}(t),
\end{split}\end{equation}
for all $t\in[0,T].$ We use mathematical induction to prove that the
inequality \eqref{e4} holds for all $k\geq1$. By virtue of lemma
\ref{Lf3} and definition of function $\mu(.)$ it follows
\begin{equation*}\begin{split}
& \mu_{1,m}(t)=E[\sup_{-\tau\leq q\leq s}
|y^{1+m}(q)-y^{1}(q)|^2]\leq \tilde{C}t=\mu_1(t).
\end{split}\end{equation*}
\begin{equation*}\begin{split}
\mu_{2,m}(t)&=E[\sup_{-\tau\leq q\leq s} |y^{2+m}(q)-y^{2}(q)|^2]\\&
\leq \hat{k}\int_{t_0}^t\kappa (E[\sup_{-\tau\leq q\leq s}
|y^{1+m}(q)-y^{1}(q)|^2])ds\\& \leq \hat{k}\int_{t_0}^t\kappa
(\mu_1(s))ds=\mu_2(t).
\end{split}\end{equation*}
By using \eqref{e3}, we derive
\begin{equation*}\begin{split}
\mu_2(t)=\hat{k}\int_{t_0}^t\kappa
(\mu_1(s))ds=\int_{t_0}^t\hat{k}\kappa
(\tilde{C}t)ds\leq\tilde{C}t=\mu_1(t).
\end{split}\end{equation*}
Hence for every $t\in[0,T_1],$ we get that $\mu_{2,m}(t)\leq
\mu_2(t)\leq \mu_1(t).$ Suppose that \eqref{e4} holds for some
$k\geq 1.$ Then we have to show that the inequality \eqref{e4} holds
for $k+1,$ as follows
\begin{equation*}\begin{split}
 \mu_{n+1,m}(t)&=E[\sup_{-\tau\leq q\leq s}
|y^{k+m+1}(q)-y^{k+1}(q)|^2]\\& \leq
\hat{k}\int_{t_0}^t\kappa(E[\sup_{-\tau\leq q\leq s}
|y^{k+m}(q)-y^{k}(q)|^2])ds\\&
=\hat{k}\int_{t_0}^t\kappa(\mu_{k,m}(s))ds \leq
\hat{k}\int_{t_0}^t\kappa(\mu_{k}(s))ds =\mu_{k+1}(t).
\end{split}\end{equation*}
And
\begin{equation*}\begin{split}\mu_{k+1}(t)=\hat{k}\int_{t_0}^t\kappa(\mu_{k}(s))ds\leq
\hat{k}\int_{t_0}^t\kappa(\mu_{k-1}(s))ds=\mu_{k}(s).
\end{split}\end{equation*}
Hence for all $t\in[0,T_1],$ $\mu_{k+1,m}(t)\leq \mu_{k+1}(t)\leq
\mu_{k}(s),$ that is, the expression \eqref{e4} holds for $k+1.$ We
observe that for $k\geq 1$, $\mu_k(t)$ is continuous and decreasing
on $t\in[0,T_1]$. By using the dominated convergence theorem, we
define the function $\mu(t)$ by
\begin{equation*}\begin{split}\mu(t)=\lim_{n\rightarrow\infty}\mu_n(t)
=\lim_{n\rightarrow\infty}\hat{k}\int_{t_0}^t\kappa(\mu_{n-1}(s))ds=\hat{k}\int_{t_0}^t\kappa(\mu(s))ds,\,\,\,\,0\leq
t\leq T_1.\end{split}\end{equation*} So,
\begin{equation*}\begin{split}\mu(t)\leq \mu(0)+\hat{k}\int_{t_0}^t\kappa(\mu(s))ds.\end{split}\end{equation*}
Hence for every $0\leq t\leq T_1$, lemma \eqref{f1} gives that
$\mu(t)=0.$ For all $t\in[0, T_1]$, \eqref{e4} follows that
$\mu_{k,m}(s)\leq\mu_{k}(s)\rightarrow 0$ as $k\rightarrow\infty$,
which gives $\hat{\mathbb{E}}|y^{k+m}(t)-y^{k}(t)|^2\rightarrow 0$
as $k\rightarrow\infty.$ Then the completeness of $L^2_G$ implies
that for all $t\in[0, T_1],$
\begin{equation*}\begin{split}  f(t,y^k_t)\rightarrow  f(t,y_t), g(t,y^k_t)\rightarrow
g(t,y_t), h(t,y^k_t)\rightarrow
h(t,y_t)\,\,\,in\,\,L^2_G\,\,\,\,as\,\,\,k\rightarrow\infty.
\end{split}\end{equation*}
Hence for all $t\in[0, T_1],$
\begin{equation*}\begin{split}\lim_{k\rightarrow\infty}y^k (t) &= \zeta (0)
 + \lim_{k\rightarrow\infty}\int_{t_0}^t  f  (s, y^{k-1}_{v})
dv  +\lim_{k\rightarrow\infty} \int_{t_0}^t  g (v,y^{n-1}_{v}) d
\langle B, B \rangle (v) + \lim_{k\rightarrow\infty}\int_{t_0}^t h
(v, y^{n-1}_{v}) dB(v),
\end{split}\end{equation*}
which implies
\begin{equation*}\begin{split} y (t) &= \zeta (0)
 +\int_{t_0}^t  f  (v, y_{v})
dv  + \int_{t_0}^t  g (v, y_{v}) d \langle B, B \rangle (v) +
\int_{t_0}^t h(v, y_{v}) dB(v),
\end{split}\end{equation*}
that is, equation \eqref{f1} with the corresponding given initial
data \eqref{i} admit a unique solution $z(t)$ on $t\in[t_0, T_1].$
By iteration, we get that equation \eqref{1}  admits a solution on
$t\in[t_0, T].$ The proof of existence is complete. To show the
uniqueness, let equation \eqref{1} admits two solutions $y(t)$ and
$z(t)$. Define $\Lambda(t)=z(t)-y(t)$,
$\hat{f}(t)=f(t,z_t)-f(t,y_t)$, $\hat{g}(t)=g(t,z_t)-g(t,y_t)$ and
$\hat{h}(t)=h(t,z_t)-h(t,y_t)$. Using the G-It$\hat{o}$ formula,
lemma \ref{l3} and lemma \ref{l4}, we derive
\begin{equation*}\begin{split}
\hat{\mathbb{E}}\Big[\sup_{0\leq v\leq t}|\Lambda (v)|^2\Big]&\leq
\hat{\mathbb{E}}\Big[\sup_{0\leq v\leq
t}\int_0^t2\langle\Lambda(v),\hat{f}(v)\rangle dv\Big]+
\hat{\mathbb{E}}\Big[\sup_{0\leq v\leq t}\int_0^t2\langle\Lambda(v),
\hat{h}(v)\rangle dB(v)\Big]\\& +\hat{\mathbb{E}}\Big[\sup_{0\leq
v\leq
t}\int_0^t[2\langle\Lambda(v),\hat{g}(v)\rangle+|\hat{h}(v)|^2d
\langle B, B \rangle (v)\Big]\\& \leq
\hat{\mathbb{E}}\int_0^t2\langle\Lambda(v),\hat{f}(v)\rangle
dv+2c_2\hat{\mathbb{E}}\Big[\int_0^t|\langle\Lambda(v)
,\hat{h}(v)\rangle|^2dv\Big]^{\frac{1}{2}}\\&+
c_1\hat{\mathbb{E}}\int_0^t[2\langle\Lambda(v),\hat{g}(v)\rangle+|\hat{h}(v)|^2]dv\\&
\leq \hat{\mathbb{E}}\int_0^t2\langle\Lambda(v),\hat{f}(v)\rangle dv
+ 2c^2_2\hat{\mathbb{E}}\int_0^t|\hat{h}(v)|^2dv
+\frac{1}{2}\hat{\mathbb{E}}\Big[\sup_{0\leq v\leq t}
|\Lambda(v)|^2\Big]\\& +
c_1\hat{\mathbb{E}}\int_0^t[2\langle\Lambda(v),\hat{g}(v)\rangle+|\hat{h}(v)|^2]dv.
\end{split}\end{equation*}
In view of assumption $A_1$ and lemma \ref{Lf3}, it follows
\begin{equation*}\begin{split}
\hat{\mathbb{E}}\Big[\sup_{0\leq v\leq t}|z(v)-y(v)|^2\Big]& \leq
2(1+2c_1+2c^2_2)\hat{\mathbb{E}}\int_0^t\kappa(\| z -
y\|_{q}^{2})dv\\& \leq
2(1+2c_1+2c^2_2)\int_0^t\kappa\Big(\hat{\mathbb{E}}\Big[\sup_{0\leq
v\leq t}| z(v) - y(v)|_{q}^{2}\Big]\Big)dv.
\end{split}\end{equation*}
Consequently lemma \ref{fl} gives
\begin{equation*}
\hat{\mathbb{E}}\Big[\sup_{0\leq v\leq t}|z(v)-y(v)|^2\Big]=0,
\end{equation*}
that is, for $t\in[0,T]$, $z(t)=y(t)$. Therefore we have $z(t)=y(t)$
holds quasi-surely for all $t\in(-\infty,T]$. The uniqueness has
been proved.
\end{proof}
\section{The $L^2_G$ and exponential estimates with weak monotonicity}
Let under the assumptions $A_1$ and $A_2$ equation \eqref{1} with
the initial data \eqref{i} has a unique solution $y(t)$ on
$t\in[0,\infty)$. We now find the $L^2_G$ estimate and then the
exponential estimate as follows.
\begin{lem}\label{Lf3} Let assumptions $A_1$ and $A_2$ hold and
$\hat{\mathbb{E}}\|\zeta\|_q^2<\infty$. Then for all $t\geq0,$
\begin{equation*}
\hat{\mathbb{E}}\Big[\sup_{-\infty< v\leq t}|y(v)|^2\Big]\leq
C_5e^{\hat{C_4}t},
\end{equation*}
where
$C_5=(2+\hat{c}b\lambda^{-1})\hat{\mathbb{E}}\|\zeta\|_q^2+\hat{c}(\tilde{K}+a)T$
and $\hat{C_4}=(2+2c_1+ b\hat{c})$, $\hat{c}=2(1+2c_1+2c^2_2)$ and
$c_1,c_2$ are positive constants.
\end{lem}
\begin{proof} Applying
the G-It$\hat{o}$ formula to $|y(t)|^2$ and using lemmas \ref{l3},
\ref{l4} for any $t\in[0, T]$, we derive
\begin{equation*}\begin{split}
\hat{\mathbb{E}}\Big[\sup_{0\leq v\leq t}|y(v)|^2\Big]&\leq
\hat{\mathbb{E}}\|\zeta\|_q^2+ \hat{\mathbb{E}}\Big[\sup_{0\leq
v\leq t}\int_0^t2\langle y(v),f(v,y_v)\rangle
dv\Big]+\hat{\mathbb{E}}\Big[\sup_{0\leq v\leq t}\int_0^t2\langle
y(v), h(v,y_v)\rangle dB(v)\Big]\\&+
\hat{\mathbb{E}}\Big[\sup_{0\leq v\leq t}\int_0^t[2\langle
y(v),g(v,y_v)\rangle+|h(v,y_v)|^2d \langle B, B \rangle (v)\Big]\\&
\leq \hat{\mathbb{E}}\|\zeta\|_q^2+ \hat{\mathbb{E}}\int_0^t2\langle
y(v),f(v,y_v)\rangle dv +
2c^2_2\hat{\mathbb{E}}\int_0^t|h(v,y_v)|^2dv
+\frac{1}{2}\hat{\mathbb{E}}\Big[\sup_{0\leq v\leq t}
|y(v)|^2\Big]\\&+ c_1\hat{\mathbb{E}}\int_0^t[2\langle
y(v),g(v,y_v)\rangle+|h(v,y_v)|^2]dv.
\end{split}\end{equation*}
By using assumptions $A_1$ and $A_2$, it follows
\begin{equation*}\begin{split}
&\hat{\mathbb{E}}\Big[\sup_{0\leq v\leq t}|y (v)|^2\Big]\\& \leq
\hat{\mathbb{E}}\|\zeta\|_q^2+2\hat{\mathbb{E}}\int_0^t[\kappa(\|y\|_q^2)+|y(v)|^2+|f(v,0)|^2]
dv + 4c^2_2\hat{\mathbb{E}}\int_0^t[\kappa(\|y\|_q^2)+|h(v,0)|^2]dv
\\&+
2c_1\hat{\mathbb{E}}\int_0^t[\kappa(\|y\|_q^2)+|y(v)|^2+|g(v,0)|^2+
\kappa(\|y\|_q^2)+|h(v,0)|^2]dv,
\end{split}\end{equation*}
which by straightforward calculations, yileds
\begin{equation*}\begin{split}
\hat{\mathbb{E}}\Big[\sup_{0\leq v\leq t}|y (v)|^2\Big]& \leq
\hat{\mathbb{E}}\|\zeta\|_q^2+\hat{c}\tilde{K}T+\hat{c}aT+2(1+c_1)\hat{\mathbb{E}}\int_0^t|y(v)|^2dv+
\hat{c}b\hat{\mathbb{E}}\int_0^t\|y\|_q^2dv,
\end{split}\end{equation*}
where $\hat{c}=2(1+2c_1+2c^2_2)$. By virtue of lemma \ref{Lf}, we
have
\begin{equation}\begin{split}\label{4} \hat{\mathbb{E}}\Big[\sup_{0\leq
v\leq t}|y(v)|^2\Big]&\leq
\hat{\mathbb{E}}\|\zeta\|_q^2+\hat{c}(\tilde{K}+a)T+2(1+c_1)\hat{\mathbb{E}}\int_0^t[\sup_{0\leq
v\leq t}|y(v)|^2]dv\\&+
\hat{c}b\hat{\mathbb{E}}\int_0^t[\|\zeta\|_q^2e^{-\lambda
v}+\sup_{0\leq v\leq t}|y(v)|^2]dv\\& \leq
\hat{\mathbb{E}}\|\zeta\|_q^2+\hat{c}(\tilde{K}+a)T+\hat{c}b\lambda^{-1}\hat{\mathbb{E}}\|\zeta\|_q^2
+(2+2c_1+ b\hat{c})\hat{\mathbb{E}}\int_0^t[\sup_{0\leq v\leq
t}|y(v)|^2]dv.
\end{split}\end{equation}
Noting that, $\hat{\mathbb{E}}\Big[\sup_{-\infty< v\leq
t}|y(v)|^2\Big]\leq
\hat{\mathbb{E}}\|\zeta\|_q^2+\hat{\mathbb{E}}\Big[\sup_{0\leq v\leq
t}|y(v)|^2\Big]$, it follows
\begin{equation*}\begin{split} \hat{\mathbb{E}}\Big[\sup_{-\infty<
v\leq t}|y(v)|^2\Big]& \leq
(2+\hat{c}b\lambda^{-1})\hat{\mathbb{E}}\|\zeta\|_q^2+\hat{c}(\tilde{K}+a)T
+(2+2c_1+ b\hat{c})\int_0^t\hat{\mathbb{E}}\Big[\sup_{-\infty< v\leq
t}|y(v)|^2\Big]dv.
\end{split}\end{equation*}
By using the Grownwall inequality, we get
\begin{equation*}
\hat{\mathbb{E}}\Big[\sup_{-\infty< v\leq t}|y^k(v)|^2\Big]\leq
C_5e^{\hat{C_4}t},
\end{equation*}
where
$C_5=(2+\hat{c}b\lambda^{-1})\hat{\mathbb{E}}\|\zeta\|_q^2+\hat{c}(\tilde{K}+a)T$
and $\hat{C_4}=2+2c_1+ b\hat{c}$. The proof is complete.
\end{proof}
\begin{thm} Let assumptions $A_1$ and $A_2$ hold. Then
\begin{equation*}
\lim_{t\rightarrow\infty}\sup \frac{1}{t}log|y(t)|\leq \beta,
\end{equation*}
where $\beta=1+c_1+ b(1+2c_1+2c^2_2)$.
\end{thm}
\begin{proof} By using the Grownwall inequality from \eqref{4}, it
follows\begin{equation}\label{5} \hat{\mathbb{E}}\Big[\sup_{0\leq
v\leq t}|y(v)|^2\Big]\leq C_6 e^{\hat{C_4} t},
\end{equation}
$C_6=(1+\hat{c}b\lambda^{-1})\hat{\mathbb{E}}\|\zeta\|_q^2+\hat{c}(\tilde{K}+a)T$
and $\hat{C_4}=2+2c_1+ b\hat{c}$. By virtue of the above result
\eqref{5}, for each $m=1,2,3,...,$ we have
\begin{equation*}
\hat{\mathbb{E}}\Big[\sup_{m-1\leq t\leq m}|y(t)|^2\Big]\leq C_6
e^{\hat{C_4} m}.\end{equation*} For any $\epsilon>0$, by using lemma
\ref{l2} we get
\begin{equation*}\begin{split}
\hat{C}\Big\{w:\sup_{m-1\leq t\leq
m}|y(t)|^2>e^{(\hat{C_4}+\epsilon)m}\Big\}&\leq
\frac{\hat{\mathbb{E}}\Big[\sup_{m-1\leq t\leq
m}|y(t)|^2\Big]}{e^{(\hat{C_4}+\epsilon)m}}\\& \leq\frac{C_6
e^{\hat{C_4} m}}{e^{(\hat{C_4}+\epsilon)m}} =C_6 e^{-\epsilon m}.
\end{split}\end{equation*}
Then by similar arguments used in theorem \ref{thm2}, we obtain
\begin{equation*}\begin{split}
\lim_{t\rightarrow\infty}\sup \frac{1}{t}log|y(t)|&\leq
\frac{\hat{C_4}+\epsilon}{2}= 1+c_1+
b(1+2c_1+2c^2_2)+\frac{\epsilon}{2},
\end{split}\end{equation*}
but $\epsilon$ is arbitrary and the above result reduces to
\begin{equation*}
\lim_{t\rightarrow\infty}\sup \frac{1}{t}log|y(t)|\leq \beta,
\end{equation*}
where $\beta=1+c_1+ b(1+2c_1+2c^2_2)$. The proof is complete.
\end{proof}

We derive the results in the phase space with fading memory
$C_q((-\infty,0];\R^n)$. However, all the results of this article
also hold in the space $BC((-\infty,0];\R^n)$ defined in
\cite{fm,m,rbs}.

\section{Acknowledgement} This work is sponsored by the Commonwealth
Scholarship Commission in the United Kingdom with project ID Number:
PKRF-2017-429.

\end{document}